\documentclass[11pt]{amsart}
\usepackage{geometry}                
\usepackage{graphicx}
\usepackage{amssymb}
\usepackage{epstopdf}
\newcommand{\Sp}{{\rm{Sp}}(2n,\mathbb Z)}
\newcommand{\SL}{{\rm{SL}}}

\newcommand{\Z}{\mathbb Z}
\newcommand{\R}{\mathbb R}
\renewcommand{\S}{\mathbb S}
\newcommand{\Q}{\mathbb Q}
\newcommand{\C}{\mathbb C}
\newcommand{\GG}{\mathbb G}
\newcommand{\G}{\Gamma}
\newcommand{\GA}{\G({\mathfrak a})}

\newcommand{\Oo}{\mathfrak o}
\newcommand{\A}{\mathfrak a}

\renewcommand{\P}{\frak p}
\newcommand{\onto}{\twoheadrightarrow}

\newtheorem{theorem}{Theorem}[section]
\newtheorem{corollary}[theorem]{Corollary}

\newtheorem{proposition}[theorem]{Proposition}

\newtheorem*{question}{Question}

\title[Arithmetic group actions on spheres and acyclic manifolds]{Actions of arithmetic groups on  homology spheres and acyclic homology manifolds}

\author[Bridson]{Martin R.~Bridson}
\address{Martin R.~Bridson\\
Mathematical Institute \\
24-29 St Giles' \\
Oxford OX1 3LB \\
U.K. }
\email{bridson@maths.ox.ac.uk}

\author[Grunewald]{Fritz Grunewald}

\author[Vogtmann]{Karen Vogtmann}
\address{Karen Vogtmann\\
Department of Mathematics\\
Cornell University\\
Ithaca NY 14853 }
\email{vogtmann@math.cornell.edu}

\thanks{Bridson is supported by an EPSRC Senior Fellowship and a Royal Society Wolfson
Research Merit Award.
Vogtmann is supported by NSF grant DMS-0204185. }



\begin{document}                                        

\begin{abstract}
We establish lower bounds on the dimensions in which arithmetic groups with torsion can act on 
acyclic manifolds and homology spheres. The bounds rely on the existence of elementary $p$-groups
in the groups concerned. In some cases, including $\Sp$, the bounds we obtain are sharp:
if $X$ is a generalized $\Z/3$-homology sphere of dimension  less than $2n-1$ or a $\Z/3$-acyclic $\Z/3$-homology manifold of dimension less than  $2n$, and if 
$n\geq 3$, then any action of $\Sp$ by homeomorphisms on $X$ is trivial; if $n=2$, then every
action of $\Sp$ on $X$ factors through the abelianization of ${\rm{Sp}}(4,\Z)$, which is  $\Z/2$.
\end{abstract}

\maketitle

\section{Introduction}

The group ${\rm{SL}}(n,\Z)$ acts faithfully by linear transformations on $\R^n$ and hence on the sphere at infinity $\S^{n-1}$,
but if $n\ge 3$ then it cannot act on  lower-dimensional spheres or Euclidean spaces. 
Indeed,  in \cite{BV}, Bridson and Vogtmann\begin{footnote}{The proof of a
similar theorem announced earlier  \cite{Parwani} is not valid.}\end{footnote} proved that if $n\ge 3$ and $d<n$, then
${\rm{SL}}(n,\Z)$ cannot act non-trivially by homeomorphisms on any contractible manifold of dimension $d$, nor on any homology sphere
of dimension $d-1$. This is an immediate corollary of a more general theorem proved in \cite{BV}: if $n\ge 3$ then
${\rm{SAut}}(F_n)$, the unique subgroup of index $2$ in the automorphism group of the free group $F_n$, cannot act
non-trivially by homeomorphisms 
 on any $\Z/2$-acyclic $\Z/2$-homology manifold of dimension less than $n$, nor any
 generalized $\Z/2$-homology sphere of dimension less than
$n-1$.  
The proof revolved around the elementary abelian $2$-groups in ${\rm{SAut}}(F_n)$. Using Smith theory, one
argues that one of a small number of involutions have to act trivially; a detailed analysis of the quotients
of ${\rm{SAut}}(F_n)$ by these involutions then allows one to conclude that the entire group has to act trivially.

Smith theory applies for any prime $p$. When $p$ is odd there is a stronger restriction on the dimension of a sphere on which an elementary abelian $p$-subgroup can act.  If $n$ is even, then ${\rm{SAut}}(F_n)$ contains an elementary
3-group of rank $n/2$, and by exploiting this one can give a considerably shorter proof of the fact that
${\rm{SAut}}(F_{2m})$  (and hence ${\rm{SL}}(2m,\Z)$) cannot act on an acyclic manifold of dimension less than $2m$
or a $\Z/3$-homology sphere of dimension less than $2m-1$ (see \cite{BV}, Theorem 4.9). 
The purpose of the
present note is to point out that the strategy of this shorter proof, combined with  ideas of 
Zimmermann \cite{Z1},  
can be used to establish similar theorems for many  arithmetic groups. In most cases, this will give only a lower bound on the dimension 
of spheres and acyclic spaces where the
group can act. But in good cases, where one has a suitably large elementary abelian $p$-subgroup for some $p$,
the bounds that one obtains can be sharp. 

One such case is that of the symplectic groups. In this case,  we prove that for $n>2$ any action of ${\rm Sp}(2n,\Z)$ on  a generalized $\Z/3$-homology sphere of dimension  less than $2n-1$ or a $\Z/3$-acyclic $\Z/3$-homology manifold of dimension less than  $2n$ must be trivial. If  $n=2$ any such action factors through the abelianization of ${\rm{Sp}}(4,\Z)$, which is $\Z/2$.  
The dimension bound is sharp,
since ${\rm{Sp}}(2n,\Z)$ does act (linearly) on $\R^{2n}$ and $\S^{2n-1}$.
This result, with a weaker conclusion for $n=2$, was obtained independently by Zimmermann (see \cite{Z}).

Our investigations are motivated by the following general question:
\begin{question} Given an affine algebraic group scheme $\GG$ defined over $\Z$ and a ring $\Oo_S$ of $S$-integers  in a number field $k,$ what is the smallest dimensional sphere (or Euclidean space) on which $\GG(\Oo_S)$ can act? In particular, are there examples of actions below the
first dimension in which  $\GG(\Oo_S)$ has a non-trivial linear representation over $\R$?
\end{question}

The bounds that we establish depend on the torsion in the groups $\GG(\Oo_S)$ and are summarized in Theorem~\ref{t:perfect}. 
 In Section \ref{s:sp} we show how everything works in the concrete setting of the symplectic group ${\rm{Sp}}(2n,\Z)$.  In Section \ref{s:lemma} we point out that the basic features of the argument work for much more general arithmetic  groups, and a slightly weaker hypothesis may be used.  We emphasize here, however, that this general argument does not eliminate the need for special arguments in low-dimensional cases, even for the symplectic group.  

In Section \ref{s:lattices} we discuss conditions on linear algebraic groups $G$ and rings of $S$-integers $\Oo_S$ that are sufficient to ensure that
the lattices $\G=G(\Oo_S)$ have the features needed for our general argument. The lower bounds on the dimension in which such groups $\Gamma$
have interesting actions depend on the existence of elementary $p$-subgroups; the larger
the rank of such a subgroup, the better the bounds one obtains.
We provide some elementary examples of such subgroups, but we leave open the question of how
to identify the largest such subgroup in general. 

\medskip

\noindent{\em{Acknowledgements.}} We thank Alex Lubotzky, Gopal Prasad and Alan Reid  for their helpful comments concerning the
material in Section 4. Most particularly, we thank Dan Segal for his notes on this material, from which we borrowed heavily. We also
thank the Institute Mittag-Leffler (Djursholm, Sweden) for its hospitality during the preparation of this manuscript.

Tragically, the second author did not survive to see this project completed. He is sorely missed for many reasons. 
Any deficiencies in the final version of this paper are the responsibility of the first and third authors alone. 

\section{The integral symplectic group}\label{s:sp}
The symplectic group ${\rm{Sp}}(2n,\mathbb Z)$ is the subgroup of ${\rm{GL}}(2n,\mathbb Z)$  consisting of matrices which preserve the standard symplectic form.  With respect to a symplectic basis $$\{a_1,\ldots,a_n,b_1,\ldots, b_n\}$$ 
for $V=\Z^{2n}$, the form is represented by the matrix 
$$
J=\begin{pmatrix} 
      0 & I \\
      -I&0\\
 \end{pmatrix}
$$
 and an integer matrix $M$ lies in   $\Sp$ if and only if  $^tMJM=J$.

 There is a subgroup $T$ of $\Sp$ isomorphic to $(\Z/3)^n$, generated by the symplectic transformations
 $$R_i\colon
 \begin{cases}
a_i\mapsto -b_i  \\
  b_i\mapsto a_i-b_i\\
  a_k\mapsto a_k \hbox{ for } k\neq i\\
  b_k\mapsto b_k \hbox{ for } k\neq i\\
  \end{cases}
  $$
  If $V_i$ is the subspace with basis $\{a_i,b_i\}$, then $R_i$ is the identity on $V^\perp_i\subset\R^{2n}$  and the matrix of $R_i$ restricted to  $V_i$   is 
$
\begin{pmatrix} 
      0& 1\\
      -1&{-1}\\
 \end{pmatrix}.
$
We remark that $T$ is also a subgroup of $\SL(2n,\Z),$ and is the same subgroup as was  used in \cite{BV}, Lemma  3.2. 

We now recall the statement of Theorem 4.7 of  \cite{BV}  for odd primes $p$, which is proved using Smith Theory.

\begin{theorem}\label{Smith}  Let $p$ be an odd prime.   If $m < 2d-1$, then $(\Z/p)^d$ cannot act effectively
by homeomorphisms on a generalized  $\Z/p$-homology sphere of dimension  $m$ or a $\Z/p$-acyclic $\Z/p$-homology manifold of dimension $m+1$.   
\end{theorem}

For $T\cong (\Z/3)^n$ this says: 

\begin{corollary}\label{trivial} Whenever $T$ acts by homeomorphisms on a generalized $\Z/3$-homology sphere of dimension less than  $2n-1$ or a 
$\Z/3$-acyclic $\Z/3$-homology manifold of dimension less than $2n$,  some non-trivial element of $T$ acts trivially.  
\end{corollary}
  
We shall use the structure of ${\rm{Sp}}(2n,\Z)$ to deduce Theorem \ref{t:sp} from Corollary \ref{trivial}.
The case $n=2$ is special: for $n\ge 3$ the group
${\rm{Sp}}(2n,\Z)$ is perfect, but ${\rm{Sp}}(4,\Z)$ maps onto ${\rm{Sp}}(4,2)$, which is 
isomorphic to the symmetric group $\Sigma_6$ and hence maps onto $\Z/2$, which acts non-trivially on the
line and the $0$-sphere.

\begin{theorem}\label{t:sp} Let $X$ be a generalized $\Z/3$-homology sphere of dimension  less than $2n-1$ or a $\Z/3$-acyclic $\Z/3$-homology manifold of dimension less than  $2n$.  
If $n\geq 3$, then any action of $\Sp$ by homeomorphisms on $X$ is trivial. If $n=2$, then every
action of $\Sp$ on $X$ factors through the abelianization of ${\rm{Sp}}(4,\Z)$, which is  $\Z/2$.
\end{theorem}

\begin{proof}    
Let  $\rho\colon\Sp\to{\rm{Homeo}}(X)$ be an action of $\Sp$ on $X$ and
let  $T\cong (\Z/3)^n$ be the subgroup defined above. By Corollary~\ref{trivial}, there is a non-trivial element $t\in T$ in the kernel of $\rho$.

The center of $\Sp$ is $\pm I$, and is not contained in $T$.  By the Margulis normal subgroup theorem, any normal subgroup of $\Sp$ is either central or has finite index.  Since the kernel of $\rho$ contains a non-central element, it must have finite index.  
$\Sp$ has the congruence subgroup property, so  the kernel of  $\rho$   contains a principal congruence subgroup $\Gamma_m$ for some $m$.   
Thus our action factors through the quotient $\Sp/\Gamma_m\cong {\rm{Sp}}(2n,\Z/m)$, and it suffices to show that any action of ${\rm{Sp}}(2n,\Z/m)$ is trivial.  

Now ${\rm{Sp}}(2n,\Z/m)\cong \prod_j{\rm{Sp}}(2n,\Z/{p_i^{e_i}})$, where $p_i$ runs over the prime divisors of $m$ (see, e.g., \cite{N} for an elementary proof of this).  If an action of ${\rm{Sp}}(2n,\Z/m)$
 is non-trivial, then its restriction to some ${\rm{Sp}}(2n,\Z/{p_i^{e_i}})$ must be non-trivial. So it is enough to prove
the theorem with
$\Sp$ replaced by ${\rm{Sp}}(2n,\Z/{p_i^{e_i}})$. 

So now we assume that $\rho$ is an action of ${\rm{Sp}}(2n,\Z/p^e)$  on $X$, and we let
$Q\leq{\rm{Homeo}}(X)$ denote its image.  $T$ maps injectively to ${\rm{Sp}}(m,\Z/p^e)$, so we
identify it with its image. As above,  some non-trivial $t\in T$  acts trivially on $X$. 

The kernel of the natural map from ${\rm{Sp}}(2n,\Z/p^e)$ onto ${\rm{Sp}}(2n,\Z/p)$ is a $p$-group. 
Except in the case $n=p=2$ the quotient ${\rm{PSp}}(2n,p)={\rm{Sp}}(2n,\Z/p)/\{\pm I\}$ is simple, so the composition factors of ${\rm{Sp}}(2n,\Z/p^e)$ are ${\rm{PSp}}(2n,p)$, $\Z/2$ and $\Z/p$.  Furthermore (except for ${\rm Sp}(4,2)$) ${\rm{Sp}}(2n,\Z/p^e)$  is perfect; thus if $Q$ is non-trivial it must map onto ${\rm PSp}(2n,p)$.  But any two maps ${\rm{Sp}}(2n,\Z/p^e)\onto {\rm{PSp}}(2n,p)$ have the same kernel, which does {\it not} contain $t$.   This contradicts the fact that the image of $t$ in $Q$ is trivial.

If $n=p=2$, then we must deal with the possibility that the image of 
$\rho$ is a non-trivial $2$-group. A convenient way to do this is to note that the following
element of order $5$ in ${\rm{Sp}}(4,\Z)$ has normal closure of index $2$:
$$
\begin{pmatrix}  
      0& 1& 0& 0\\
      -1& 0& 1& {-1}\\
      0& 1& 0& 1\\
     0& 0& -1& 1\\
 \end{pmatrix}
$$ 
It follows that the only non-trivial
finite 2-group onto which ${\rm{Sp}}(4,\Z)$ can map is $\Z/2$.

(This matrix is the image of an element $\mu$ of order $10$ in the mapping class group ${\rm{Mod}}_2$ of a genus 2 surface, under the natural map ${\rm{Mod}}_2\to {\rm{Sp}}(4,\Z)$.  We refer the reader to \cite{HK} for a  geometric argument  that  the normal closure of  $\mu$ in ${\rm{Mod}}_2$ has index $2$.)
\end{proof}

\section{A General Argument}\label{s:lemma}

The argument for ${\rm{Sp}}(2n,\Z)$ used several features of this group which are common in higher-rank arithmetic groups.   ${\rm{Sp}}(2n)$ is an example of a group scheme defined over $\Z$, and it  is convenient to use this language to phrase a general question about actions of such groups on spheres and Euclidean spaces.

Recall that an {\it affine algebraic group scheme $\GG$ defined over $\Z$} is a functor from the category of rings to groups which is {\it represented} by some finitely-generated ring $A$; this means that  there is a natural bijection between $\GG(R)$ and ${\rm{Hom}}(A,R)$ for any ring $R$.  
As a simple example, consider the functor $\GG={\rm{SL}}_2$; it is an affine group scheme over $\Z$ represented by the ring   $A=\Z[x_1,x_2,x_3,x_4]/\langle x_1x_4-x_2x_3-1\rangle$.  For a  very readable introduction to affine group schemes, see \cite{Waterhouse}.

\begin{question} Given an affine algebraic group scheme $\GG$ defined over $\Z$ and a ring $\Oo_S$ of $S$-integers  in a number field $k,$ what is the smallest dimensional sphere (or Euclidean space) on which $\GG(\Oo_S)$ can act? In particular, are there examples of actions below the
first dimension in which  $\GG(\Oo_S)$ has a non-trivial linear representation over $\R$?
\end{question}

The answer to this question will depend on both the group scheme $\GG$ and on  $\Oo_S$.  In our answer for $\GG={\rm{Sp}}(2n)$ and $\Oo_S=\Z$ we used Smith theory to find a non-central element in ${\rm{Sp}}(2n,\Z)$ which acted trivially; this allowed us to reduce the problem to  actions of ${\rm{Sp}}(2n,\Z/p^e)$.  We then began over, with a Smith theory argument for ${\rm{Sp}}(2n,\Z/p^e)$.  Instead of starting over, we could have hypothesized the existence of one element $t\in{\rm{Sp}}(2n,\Z)$ which acts trivially and which projects to a non-central element in each ${\rm{Sp}}(2n,\Z/p^e)$.  A slightly more complicated argument, which we give below, can then be used to show that  $t$ normally generates ${\rm{Sp}}(2n,\Z)$ so the whole action is trivial.  The advantage of this  new  argument is that it applies whenever   $\G=\GG(\Oo_S)$   enjoys the following properties:

\begin{enumerate}

\item {\bf  Normal Subgroup Property.} 
Every normal subgroup of $\G$ is either central or has finite index 
 
\item  {\bf Weak Congruence Subgroup Property.}  For every finite-index normal subgroup $K<\G$,
the pre-image in $\G$ of the centre of $\G/K$
 contains a principal congruence subgroup $\G(\A)=\ker (\G\to \GG(\Oo_S/\A))$.  

\item  {\bf Very Strong Approximation.}  The map $\G\to \GG(\Oo_S/\A)$ is surjective for all ideals $\A$.  

\item {\bf Quasi-Solvable Quotients.}  
For every prime ideal $\P$ in $\Oo_S$ the quotient of $\GG(\Oo_S/\P)$ by its centre is centreless with only solvable proper quotients
(e.g. simple).

\end{enumerate}

\begin{proposition}\label{Normal}  Let $\G$ satisfy properties (1)-(4) above and
suppose that $\G$ contains an element  $t$ whose image in $\GG(\Oo_S/\P)$ is non-central for every prime ideal $\P$.
Then the quotient of $\G$ by the normal closure of $t$ is finite and solvable. In particular, if $\G$
is perfect then the normal closure of $t$ is equal to $\G$.
\end{proposition}

\begin{proof} Let $K$ be the subgroup of $\G$ normally generated by $t$,
let  $N$ be the pre-image in $\G$ of the centre of $\G/K$, and let $f\colon \G\onto  \G/N := Q$
be the quotient map. 
 
Since $N$ contains the non-central element $t$, it has finite index, by $(1)$.   
 The Congruence Subgroup Property (2) tells us that  $N$ 
 must contain a congruence subgroup $\GA=\ker (\G\to \GG(\Oo_S/\A))$, and   (3) tells us that $\G\to \GG(\Oo_S/\A)$ is onto. Thus $f$ induces an epimorphism
$$\overline f\colon  \GG(\Oo_S/\A)\onto Q.$$

Since $\Oo_S$ is a Dedekind domain, $\A$ factors uniquely as a product of powers of prime ideals,
say $\A=\prod \P_i^{e_i}$. By the Chinese Remainder Theorem,
 $\GG(\Oo_S/\A)= \prod \GG(\Oo_S/\P_i^{e_i})$. We will be done if we can argue that the restriction of $\overline f$
 to each $\GG(\Oo_S/\P_i^{e_i})$ has solvable image.
 
An elementary calculation shows that the kernel  of the natural map
$\GG(\Oo_S/\P^e)\to\GG(\Oo_S/\P)$ is a $p$-group, where $p$ is the characteristic of $\Oo_S/\P$. Thus the kernel of $$\prod \GG(\Oo_S/\P_i^{e_i})
\to \prod \GG(\Oo_S/\P_i)$$ is a nilpotent group, which we denote $P$.

Now consider the following commutative diagram of epimorphisms:
$$
\begin{matrix}
\prod\GG(\Oo_S/\P_i^{e_i})&\xrightarrow{\overline{f}}&Q\\
\downarrow &&\downarrow\\
\prod \GG(\Oo_S/\P_i)& \xrightarrow{} & Q/\overline{f}(P) 
\end{matrix}
$$
 Let $(t_1,\dots,t_n)$ denote the image of $t$ in $\prod \GG(\Oo_S/\P_i)$. By hypothesis, each $t_i$ is non-central
 in $\GG(\Oo_S/\P_i)$.
On the other hand, the image of $t$ in $Q$ is trivial, by hypothesis, so the image of each $t_i$ in
$Q/\overline{f}(P)$ is central by the virtue of the following elementary observation:

\smallskip
{\em{Let $\phi:G_1\times G_2\to A$ be an epimorphism of groups. If $\phi(g_1,g_2)=1$
then $\phi(g_1,1)$ and $\phi(1,g_2)$ are central in $A$.}}

\smallskip

Property (4) now implies that  the image of $\GG(\Oo_S/\P_i)$ in $Q/\overline{f}(P)$ is  solvable,
and hence $Q/\overline{f}(P)$, a product of commuting solvable groups, is solvable.
Since $\overline{f}(P)$ is nilpotent, we conclude that $Q$ is solvable and hence so is
its central extension $\G/K$, as required.
\end{proof}

We combine this proposition with Theorem~\ref{Smith} to obtain restrictions on the possible actions of $\G$.

\begin{theorem}\label{t:perfect}
Let $p$ be an odd prime,   let  $\Gamma=\GG(\Oo_S)$, as above, be a perfect
group with properties (1)-(4), and consider an action
of $\G$ on a generalized $m$-dimensional $\Z/p$-sphere or  $(m + 1)$-dimensional   $\Z/p$-acyclic homology manifold over $\Z/p$. Suppose that $\Gamma$ contains an elementary abelian $p$-group $T$ of rank $\lfloor m/2\rfloor+1$ whose projection to $\GG(\Oo_S/\P)$ intersects the centre trivially for each
prime ideal $\P$ in $\Oo_S$.  Then the action is trivial. 
\end{theorem}

\begin{proof}  Theorem~\ref{Smith}  gives us an element $t\in T\subset\Gamma$ which acts trivially.  Then Proposition~\ref{Normal} allows us to conclude that the entire group $\Gamma$ acts trivially.
\end{proof}

The subgroup $T$ used in our previous proof for ${\rm{Sp}}(2n,\Z)$ satisfies the hypothesis of Theorem~\ref{t:perfect}, so this gives a slightly different proof for 
${\rm{Sp}}(2n,\Z)$ when $n\ge 3$.

\section{Arithmetic groups}\label{s:lattices}

 Let $G$ be an algebraic group defined over $\Q$.  We fix an embedding $G\hookrightarrow {\rm{GL}}_n$.  This means that we identify $G$ with a set of invertible matrices whose entries satisfy specific polynomial equations with coefficients in $\Q.$ 
 These equations make sense for any field $k$ containing $\Q$, so we can define $G(k)$ to be ${\rm{GL}}(n,k)\cap G$.  If $k$ is an algebraic number field with ring of integers $\Oo$ and $S$ a set of valuations,  we may then define $G(\Oo_S)$ to be ${\rm{GL}}(n,\Oo_S)\cap G(k).$

 In this section we discuss conditions on  $G, k$ and $S$  which will guarantee that $\G=G(\Oo_S)$ has the properties we need for Proposition~\ref{Normal}.    We assume throughout that $G$ is connected (in the Zariski topology), simply connected (has no proper \'etale covers), and absolutely simple (i.e. simple when considered as an algebraic group over the algebraic closure of the defining field). Examples of such groups include the classical special linear and symplectic groups, as well as commutator subgroups of the orthogonal groups.  We also assume that  $\G$ is infinite throughout.
  
\subsection{Normal subgroups} The first property in the list   from Section~\ref{s:lemma}  is the Normal Subgroup Property, which was established for a large class of algebraic groups by Margulis \cite{Margulis}.  In particular, under the conditions fixed above, the Normal Subgroup Property holds for $\G$ as long as the dimension of a maximal $k$-split torus in $G$ is at least 2 ($rank_k(G)\geq 2$) (\cite{Margulis},  quoted in \cite{PR} as Theorem 9.9).

\subsection{Congruence Subgroups}
Sufficient conditions for ensuring Property (2) (the Congruence Subgroup property) are given at length in \cite{PR}, chapter 9.  Our standing assumptions on $G$ are sufficient together with:  
\begin{itemize}
\item $G$ has $k$-rank at least 2,
\item if $G$ is one of $B_2$ or $G_2$ then $S$ contains all prime divisors of $2$ with residue degree 1.
\end{itemize}

 \subsection{Strong Approximation}
The third property on our list  is the most problematic.     For $G={\rm{SL}}_n$ the statement is classical and can be found, for example, in Bourbaki (\cite{Bourbaki}, Chapter VII, section 2, part 4 on Dedekind domains).  The proof uses the fact that ${\rm{SL}}(n,\Oo_S)$ and ${\rm{SL}}(n,\Oo_S/\A)$  are generated by elementary matrices $I+\lambda E_{i,j}$ together with the Chinese Remainder Theorem.   

For most linear algebraic groups $G$ one cannot expect $\G\to G(\Oo_S/\A)$ to be surjective for all $\A$.  One reason is that the equations defining $G$  may change drastically when reduced modulo a prime, or even become completely trivial; in fact we should write $G^{(\A)}(\Oo_S/\A)$ to indicate that we are considering the defining equations for $G$ modulo $\A$ in the image.  However, in our proof we cannot afford to exclude any primes because we have no control over which ideal $\A$ will be given to us by the  Congruence Subgroup Property.    It turns out that it is enough to  assume that 
\begin{itemize}
\item For every prime $\P$ in $\Oo_S$, the map $G(\Oo_\P)\to G^{(\P)}(\Oo/\P)$ is surjective.
\end{itemize}
(Here $\Oo_\P$ is the localization of $\Oo$ at the prime $\P$, i.e. the ring obtained from $\Oo$ by inverting all primes except $\P$, and $\Oo/\P=\Oo_S/\P=\Oo_\P/\P$.)  This assumption is satisfied, for example, by all of the Chevalley groups (see \cite{Carter}).

Under this assumption, Strong Approximation implies the needed surjectivity.  This argument goes as follows. The Strong Approximation as stated in \cite{PR}, Theorem 7.12, can be rephrased as: 

\begin{theorem}\label{thm:SA}   Let $\P_{1},\ldots,\P_{r}$ be primes not in $S$.
Let $e_{i}\in\mathbb{N}$ and put 
\[
K_{i}=\left\{  g\in G(\Oo_{\P_{i}})\mid g\equiv I_{n}~(\operatorname{mod}\P_{i}^{e_{i}})\right\}  .
\]
Embed $\G$ diagonally in $G(\Oo_{\P_{1}})\times
\cdots\times G(\Oo_{\P_{r}})$. Then 
\[
\G\cdot(K_{1}\times\cdots\times K_{r})=G(\Oo_{\P_{1} })\times\cdots\times G(\Oo_{\P_{r}}).
\]
\end{theorem}

When $G={\rm{SL}}_n$, $\Oo=\mathbb Z$, $S=\emptyset$ and we have only one prime $p$,  
the theorem reduces to the statement ${\rm{SL}}(n,\Z)\cdot K={\rm{SL}}(n,\Z_p)$, i.e. every determinant 1 matrix with entries in $\Z_p$ can be written as a product of a matrix with entries in $\Z$ times a matrix  congruent to the identity mod $p^e$ for any natural number $e$. 
 
Now let $\mathfrak{a}=\prod\P_{i}^{e_{i}}$ be an arbitrary ideal of $\Oo_S$. It follows that the image of $\Gamma$ in ${\rm{GL}}_n(\Oo/\mathfrak{a})$ is the same as the image of $G(\Oo_{\P_{1}})\times\cdots\times G(\Oo_{\P_{r}})$.   Using the Chinese Remainder theorem and the assumption that the maps $G(\Oo_\P)\to G^{(\P)}(\Oo/\P)$ are surjective, one shows that this latter image is just $G(\Oo/\mathfrak{a})$. (This includes showing
 that the surjectivity of $G(\Oo_\P)\to G^{(\P)}(\Oo/\P)$ implies the surjectivity of  $G(\Oo_\P)\to G^{(\P)}(\Oo/\P^e)$, which
is explained in Lemma 5 on page 393 of \cite{LS} in the case $\Oo_\P=\Z_p$.)

\subsection{Simplicity} 
If $G$ is a Chevalley group and $G\hookrightarrow GL_n$ is the standard representation given by the Chevalley basis (see \cite{Carter2}, Chapter 4), then $G^{(\P)}(\Oo/\P)$ is almost always simple modulo its centre and perfect \cite{Carter}.  The only exceptions are the groups in the following list:
$$ A_1(\mathbb F_2), A_1(\mathbb F_3), ^2\!A_2(\mathbb F_2), B_2(\mathbb F_2), ^2\!B_2(\mathbb F_2), G_2(\mathbb F_2), ^2\!F_4(\mathbb F_2), 
^2\!G_2(\mathbb F_3).$$  

\subsection{Some good groups} 
All of our conditions are satisfied if   
$G$ is a Chevalley group of rank at least $2$, $G\hookrightarrow GL_n$ is the standard representation, and if $G$ is  $B_2$ or $G_2$ then $S$ contains all prime divisors of $2$ with residue degree $1$.
One can extract from the references quoted above that there are also many other groups with the properties we require.

\section{Special linear and symplectic groups} Let $k$ be a number field. 
In this section we restrict to special linear and symplectic groups over rings $\Oo_S$ of $S$-integers.  For $\Oo_S=\mathbb Z$, of course, we have already established optimal bounds for actions on homology spheres and acyclic homology manifolds.

\subsection{Linear actions.}     Let $k$ and $\Oo_S$ be as above. Theorem~\ref{Smith} tells us that
whenever it acts on $\S^{n-2}$ or $\R^{n-1}$,  the subgroup ${\rm{SL}}(n,\Z)\leq {\rm{SL}}(n,\Oo_S)$ must act trivially.  Therefore the kernel of the action contains elements which are not central, and are not central mod $\mathfrak p$ for any prime $\mathfrak p$.  Thus  the entire action is trivial by Proposition~\ref{Normal}.  The same observation shows that ${\rm{Sp}}(2n,\Oo_S)$ cannot act on $\R^{2n-1}$ or $\S^{2n-2}$.  

If $k$ is totally real, then ${\rm{SL}}(n,\Oo_S)$ and ${\rm{Sp}}(2n,\Oo_S)$ embed in $\SL(n, \R)$ and $\SL(2n, \R)$, respectively, so the above
observations provide a complete answer to Question \ref{q:smallest} in these cases.
But if  $k$ is not totally real,   then we have only that
${\rm{SL}}(n,\Oo_S)$ acts on $\C^n$ (real dimension $2n$) and $\S^{2n-1}$; and similarly,
 ${\rm{Sp}}(2n,\Oo_S)$ acts non-trivially on $\R^{4n}$  and $\S^{4n-1}$. Thus, in this case, we have not determined the least dimension of a 
 contractible manifold or sphere on which these groups act non-trivially.

\subsection{Roots of unity.}  Suppose that $\Oo$ contains a  $p$-th root of unity  $\eta$ for some odd prime $p$.  The subgroup
$D$ of ${\rm{SL}}(n,\Oo_S)$ generated by diagonal matrices with powers of $\eta$ on the diagonal is a copy of $(\Z/p)^{n-1}$.   By Theorem~\ref{Smith}, 
whenever ${\rm{SL}}(n,\Oo_S)$ acts on $\R^{2n-3}$ or $\S^{2n-4}$, some element of $D$ acts trivially.  
$D$ intersects the centre of  ${\rm{SL}}(n,\Oo_S)$ if and only if $n$ is divisible by $p$.  Moreover,
if $p$ does not divide $n$, then no element of $D$ become central modulo any prime $\mathfrak p$. 

If $p$ does not divide $n$,  we conclude that any action of ${\rm{SL}}(n,\Oo_S)$ on a Euclidean space of dimension less than $2n-2$ is trivial, i.e. ${\rm{SL}}(n,\Oo_S)$ can't act on $\R^{2n-3}$.    Since there are actions on $\R^{2n}$ and $\S^{2n-1}$, this leaves only two dimensions in which there may or may not be actions.

If $p$ divides $n$,  then ${\rm{SL}}(n,\Oo_p)$ contains a non-central copy of $(\Z/p)^{n-2}$, so ${\rm{SL}}(n,\Oo_p)$ can't act on $\R^{2n-5}$.

\section{Questions}

The Smith theory techniques underlying this note rely on the presence of  $p$-torsion.  To determine the scope  of their
applicability, one needs to address the following natural question:

\begin{question}\label{q:smallest}  Given an affine group scheme $\GG$, a ring of $S$-integers $\Oo_S$, and a rational prime $p$,
 what is the rank of the largest elementary abelian $p$-group in $\GG(\Oo_S)$? 
\end{question}

There are obvious bounds coming from
the observation that the generators of any such elementary abelian $p$-subgroup are all simultaneously diagonalizable over $\C$ (as can be seen easily from considering their Jordan forms), so the rank is at most the dimension   of a maximal torus in $G(\C)$.  For example, in ${\rm{SL}}_n$, it is at most $n-1$.

Smith theory techniques tell us  nothing in the absence of torsion, and little is known about how torsion-free lattices in semi-simple Lie groups 
of higher rank might act on contractible manifolds or spheres. Such problems fit naturally into the context of
the Zimmer Programme \cite{zimmProg}, but the techniques developed until now are not well-adapted to actions by homeomorphisms.

\bibliographystyle{siam}

\end{document}